\documentclass[12]{amsart}
\usepackage{amsfonts}
\usepackage{mathrsfs}
\usepackage{amsmath}
\usepackage{amssymb}
\usepackage{graphicx}
\newtheorem{thm}{Theorem}[section]
\newtheorem{cor}[thm]{Corollary}
\newtheorem{lem}[thm]{Lemma}

\newtheorem{prop}[thm]{Proposition}

\theoremstyle{definition}
\newtheorem{defn}[thm]{Definition}

\numberwithin{equation}{section}

\begin{document}
\vbox{\vskip 1cm}

\title{\small A generalization of supplemented modules }
\maketitle
\begin{center}
Yongduo Wang
\vskip 4mm Department of Applied Mathematics, Lanzhou University of Technology\\
Lanzhou 730050, P. R. China\\
 E-mail: ydwang@lut.cn

\end{center}

\bigskip
\leftskip10truemm \rightskip10truemm \noindent {\bf Abstract}
 Let $M$ be a left module over a ring $R$ and $I$ an ideal of $R$. $M$ is called an $I$-supplemented module (finitely
$I$-supplemented module) if for every submodule (finitely generated
submodule ) $X$ of $M$, there is a submodule $Y$ of $M$ such that
$X+Y=M$, $X\cap Y\subseteq IY$ and $X\cap Y$ is PSD in $Y$. This
definition generalizes supplemented modules and
$\delta$-supplemented modules. We characterize
 $I$-semiregular, $I$-semiperfect and $I$-perfect rings which are defined by
 Yousif and Zhou [15] using $I$-supplemented modules.
 Some well known results are obtained as corollaries.\\


\leftskip0truemm \rightskip0truemm
\bigskip

\baselineskip=20pt

\section{\bf Introduction and Preliminaries}

\noindent It is well known that supplemented modules play an
important role in characterizing semiperfect, semiregular and
perfect rings. Recently, some authors had worked with various
extensions of these rings (see for examples [2, 8, 10, 15, 16]). As
generalizations of semiregular rings, semiperfect rings and perfect
rings, the notions of $I$-semiregular rings, $I$-semiperfect rings
and $I$-perfect rings were introduced by Yousif and Zhou [15]. Our
purposes of this paper is to chracterize $I$-semiregular rings,
$I$-semiperfect rings and $I$-perfect rings by defining
$I$-supplemented modules.

Let $R$ be a ring and $I$ an ideal of $R$, $M$ a module and $S\leq
M$. $S$ is called \emph{small} in $M$ (notation $S\ll M$) if $M\neq
S + T$ for any proper submodule $T$ of $M$.  As a proper
generalization of small submodules, the concept of $\delta$-small
submodules was introduced by Zhou[16]. $N$ is said to be
\emph{$\delta$-small} in $M$ if, whenever $N+X=M$, $M/X$ singular,
we have $X=M$. $\delta(M)=Rej_{M}(\wp)=\cap\{N\leq M\mid
M/N\in\wp\}$ , where $\wp$ be the class of all singular simple
modules. Let $N, L\leq M$. $N$ is called a \emph{supplement} of $L$
in $M$ if $N + L = M$ and $N$ is minimal with respect to this
property. Equivalently, $M = N + L$ and $N\cap L\ll N$. A module $M$
is called \emph{supplemented} if every submodule of $M$ has a
supplement in $M$. A module $M$ is said to be \emph{lifting} if for
any submodule $N$ of $M$, there exists a direct summand $K$ of $M$
such that $K\leq N$ and $N/K\ll M/K$, equivalently, for every
submodule $N$ of $M$, $M$ has a decomposition with $M=M_1\oplus
M_2$, $M_1\leq N$ and $M_2\cap N$ is small in $M_2$. $N$ is called a
\emph{$\delta$-supplement} [5] of $L$ if $M=N+L$ and $N\cap
L\ll_{\delta}N$. $M$ is called a \emph{$\delta$-supplemented module}
if every submodule of $M$ has a $\delta$-supplement. A module $M$ is
said to be \emph{$\delta$-lifting} [5] if for any submodule $N$ of
$M$, there exists a direct summand $K$ of $M$ such that $K\leq N$
and $N/K\ll_{\delta} M/K$, equivalently, for every submodule $N$ of
$M$, $M$ has a decomposition with $M=M_1\oplus M_2$, $M_1\leq N$ and
$M_2\cap N$ is $\delta$-small in $M_2$.  An element $m$ of $M$ is
called \emph{$I$-semiregular} [2] if there exists a decomposition
$M=P\oplus Q$ where $P$ is projective, $P\subseteq Rm$ and $Rm\cap
Q\subseteq IM$. $M$ is called \emph{an $I$-semiregular module} if
every element of $M$ is $I$-semiregular. $R$ is called
\emph{$I$-semiregular} if $_{R}R$ is an $I$-semiregular module. Note
that $I$-semiregular rings are left-right symmetric and $R$ is
($\delta-$) semiregular if and only $R$ is ($\delta(_{R}R)-)$
$J(R)$-semiregular. $M$ is called \emph{an $I$-semiperfect module}
[10] if for every submodule $K$ of $M$, there is a decomposition
$M=A\oplus B$ such that $A$ is projective, $A\subseteq K$ and $K\cap
B\subseteq IM$. $R$ is called \emph{$I$-semiperfect} if $_{R}R$ is
an $I$-semiperfect module. Note that $R$ is ($\delta-$) semiperfect
if and only $R$ is ($\delta(_{R}R)-)$ $J(R)$-semiperfect. $R$ is
called \emph{a left $I$-perfect ring} [15] if, for any submodule $X$
of a projective module $P$, $X$ has a decomposition $X=A\oplus B$
where $A$ is a direct summand of $P$ and $B\subseteq IP$. By [10,
Proposition 2.1], $R$ is a left $I$-perfect ring if and only if
every projective module is an $I$-semiperfect module. For other
standard definitions we refer to [3, 4, 14].

In this note all rings are associative with identity and all modules
are unital left modules unless specified otherwise. Let $R$ be a
ring and $M$ a module. We use $Rad(M)$, $Soc(M)$, $Z(M)$ to indicate
the Jacobson radical, the socle, the singular submodule of $M$
respectively. $J(R)$ is the radical of $R$ and $I$ is an ideal of
$R$.

\section{\bf PSD submodules and $I$-supplemented modules}

In this section, we give some properties of PSD submodules and use
PSD submodules to define (finitely) $I$-supplemented modules and
$I$-lifting modules which are generalizations of some well-known
supplemented modules and lifting modules. Some properties of
$I$-supplemented modules are discussed. We begin this section with
the following definitions.

\begin{defn} ([12]) Let $I$ be an ideal of $R$ and $N\leq M$. $N$ is PSD in
$M$ if there exists a projective summand $S$ of $M$ such that $S\leq
N$ and $M=S\oplus X$ whenever $N+X=M$ for any submodule $X\leq M$.
$M$ is PSD for $I$ if any submodule of $IM$ is PSD in $M$. $R$ is a
left PSD ring for $I$ if any finitely generated free left $R$-module
is PSD for $I$.
\end{defn}

\begin{lem} Let $M$ and $N$ be modules.
\begin{enumerate}
\item If $K$ is PSD in $M$ and $f: M\rightarrow N$ is an
epimorphism, then $f(K)$ is PSD in $N$.
\item If $L\leq N\leq M$ and $L$ is PSD in $N$, then $L$ is PSD in $M$.

\item If $L\leq N\leq M$ and $N$ is PSD in $M$, then $L$ is PSD in
$M$.
\item Let $M=M_{1}\oplus M_{2}$. If $N_{1}$ is PSD in $M_{1}$
and $N_{2}$ is PSD in $M_{2}$, then $N_{1}\oplus N_{2}$ is PSD in
$M$.

\item Let $N$ be a direct summand of $M$ and $A\leq N$. Then $A$ is
PSD in $M$ if and only if $A$ is PSD in $N$.

\end{enumerate}
\end{lem}
\begin{proof} $(1)$ Let $f(K)+L=N, L\leq N$. Then $K+f^{-1}(L)=M$.
Since $K$ is PSD in $M$, there is a projective summand $H$ of $M$
with $H\leq K$ such that $H\oplus f^{-1}(L)=M$. So $f(H)\oplus L=N,
f(H)\subseteq f(K)$. It is easy to see that $f(H)\cong H$ is
projective.

$(2)$ Let $M=L+X$, $X\leq M$. Then $N=L+(N\cap X)$. Since $L$ is PSD
in $N$, there is a projective summand $H$ of $N$ with $H\leq L$ such
that $N=H\oplus (N\cap X)$, and hence $L=H\oplus (L\cap X)$. So
$M=H\oplus X$.

$(3)$ Let $M=L+K, K\leq M$, then $M=N+K$. Since $N$ is PSD in $M$,
there is a projective summand $H$ of $M$ with $H\leq N$ such that
$M=H\oplus K$, and hence $M/K\cong H$ is projective. Thus the
natural epimorphism $f: L\rightarrow M/K$ splits and $Kerf=L\cap K$
is a direct summand of $L$. Write $L=(L\cap K)\oplus Q, Q\leq L$, we
have $M=Q\oplus K$. The rest is obvious.

$(4)$ and $(5)$ See [12].

\end{proof}

\begin{prop} Let $M$ be a module and $N\leq M$.
\begin{enumerate}
\item $N\ll M$ if and only if $N\subseteq Rad(M)$, $N$ is PSD in
$M$.
\item $N\ll_{\delta} M$ if and only if $N\subseteq \delta(M)$, $N$ is PSD in
$M$.
\end{enumerate}
\end{prop}
\begin{proof}

$(1)$ $``\Rightarrow"$ is clear.

$``\Leftarrow"$ Let $M=N+L, L\leq M$. Since $N$ is PSD in $M$, there
is a projective summand $H$ of $M$ with $H\subseteq N\subseteq
Rad(M)$ such that $M=H\oplus L$. So $Rad(H)\oplus
Rad(L)=Rad(M)=H\oplus Rad(L)$. Thus $Rad(H)=H$. Since $H$ is
projective, $H=0$, and hence $L=M$.

$(2)$ $``\Rightarrow"$ is clear.

$``\Leftarrow"$ Let $M=N+L, L\leq M$. Since $N$ is PSD in $M$, there
is a projective summand $H$ of $M$ with $H\subseteq N\subseteq
\delta(M)$ such that $M=H\oplus L$. So $\delta(H)\oplus
\delta(L)=\delta(M)=H\oplus \delta(L)$. Thus $\delta(H)=H$. Since
$H$ is projective, $H$ is semisimple by [10, Proposition 2.13]. Thus
$N\ll_{\delta} M$ by [16, Lemma 1.2].

\end{proof}

\begin{cor} Let $M$ be a module. Then
\begin{enumerate}
\item $M$ is ($\delta$-) supplemented if and only if for every submodule
$X$ of $M$, there is a submodule $Y$ of $M$ such that $X+Y=M$,
$X\cap Y\subseteq (\delta (Y))$ $Rad(Y)$ and $X\cap Y$ is PSD in
$Y$.
\item $M$ is ($\delta$-) lifting if and only if for every submodule
$X$ of $M$, there is a decomposition $M=A\oplus B$ such that
$A\subseteq X$ and $X\cap B\subseteq (\delta (B))$ $Rad(B)$ and
$X\cap B$ is PSD in $B$.

\end{enumerate}
\end{cor}

\begin{defn} Let $R$ be a ring and $I$ an ideal of $R$, $M$ a
module. $M$ is called an $I$-supplemented module (finitely
$I$-supplemented module) if for every submodule (finitely generated
submodule ) $X$ of $M$, there is a submodule $Y$ of $M$ such that
$X+Y=M$, $X\cap Y\subseteq IY$ and $X\cap Y$ is PSD in $Y$. In this
case, we call $Y$ is an $I$-supplement of $X$ in $M$. $M$ is called
$I$-lifting if for every submodule $X$ of $M$, there is a
decomposition $M=A\oplus B$ such that $A\subseteq X$ and $X\cap
B\subseteq IB$ and $X\cap B$ is PSD in $B$.

\end{defn}

\begin{thm} Consider the following statements for a module $M$.

\begin{enumerate}
\item $M$ is a $J(R)$-supplemented module (a $\delta(_{R}R)$-supplemented module,
respectively).
\item $M$ is a supplemented module ( a $\delta$-supplemented module,
respectively).

\end{enumerate}

then $``(1)\Rightarrow (2)"$, $``(2)\Rightarrow (1)"$ if $M$ is
projective or $R$ satisfies $J(R)M=Rad(M)$  $(\delta(R)M=\delta
(M))$ for any module $M$ over $R$.
\end{thm}
\begin{proof} $``(1)\Rightarrow (2)"$ By Proposition 2.3.

$``(2)\Rightarrow (1)"$ Let $M$ be a supplemented module. Then for
every submodule $X$ of $M$, there is a submodule $Y$ of $M$ such
that $X+Y=M$ and $X\cap Y\ll Y$. Since $M$ is projective, $Y$ is a
direct summand of $M$, and hence $Y$ is projective. It is clear that
$X\cap Y\subseteq Rad(Y)=J(R)Y$ and $X\cap Y$ is PSD in $Y$. (Let
$M$ be a $\delta$-supplemented module. Since $M$ is projective, $M$
is $\delta$-lifting. Thus for every submodule $X$ of $M$, there is a
direct summand $Y$ of $M$ such that $M=X+Y$ and $X\cap Y\ll_{\delta}
Y$. The rest is obvious.) When $R$ satisfies $J(R)M=Rad(M)$
$(\delta(R)M=\delta (M))$ for any module $M$ over $R$, the proof is
similar.

\end{proof}

Similar to the proof of Theorem 2.6, we have the following.

\begin{thm} Consider the following statements for a module $M$.

\begin{enumerate}
\item $M$ is a finitely $J(R)$-supplemented module (a finitely $\delta(_{R}R)$-supplemented module,
respectively).
\item $M$ is a finitely supplemented module ( a finitely $\delta$-supplemented module,
respectively).

\end{enumerate}

then $``(1)\Rightarrow (2)"$, $``(2)\Rightarrow (1)"$ if $M$ is
projective or $R$ satisfies $J(R)M=Rad(M)$  $(\delta(R)M=\delta
(M))$ for any module $M$ over $R$.

\end{thm}

\begin{thm} Consider the following statements for a module $M$.

\begin{enumerate}
\item $M$ is a $J(R)$-lifting module (a $\delta(_{R}R)$-lifting module,
respectively).
\item $M$ is a lifting module ( a $\delta$-lifting module,
respectively).

\end{enumerate}

then $``(1)\Rightarrow (2)"$, $``(2)\Rightarrow (1)"$ if $M$ is
projective or $R$ satisfies $J(R)M=Rad(M)$  $(\delta(R)M=\delta
(M))$ for any module $M$ over $R$.

\end{thm}

We know that if a ring $R$ is left ($\delta-$)semiperfect ring, then
$(\delta(R)M=\delta (M))$ $J(R)M=Rad(M)$ for any module $M$ over
$R$. So $``(1)\Leftrightarrow (2)"$ in Theorem 2.6, 2.7 and 2.8 if
$R$ is left ($\delta-$)semiperfect ring.

\begin{lem} Let $M$ be a module and $K, L, H\leq M$. If $K$ is an
$I$-supplement of $L$ in $M$, $L$ is an $I$-supplement of $H$ in
$M$, then $L$ is an $I$-supplement of $K$ in $M$.
\end{lem}
\begin{proof} Let $M=K+L=L+H$, $K\cap L\subseteq IK, L\cap
H\subseteq IL$ and $K\cap L$ be PSD in $K$, $L\cap H$ PSD in $L$. We
only show that $K\cap L\subseteq IL$ and $K\cap L$ is PSD in $L$. It
is easy to see that $K\cap L\subseteq IK\cap L$. Let
$l=\Sigma_{i=1}^{n}p_{i}k_{i} \in IK\cap L$, $p_{i}\in I, k_{i}\in
K$ and $k_{i}=l'_{i}+h_{i} (i=1, 2, \cdot\cdot\cdot, n)$, $l'_{i}\in
L, h_{i}\in H$. Since $L\cap H\subseteq IL$, $l\in IL$, and hence
$K\cap L\subseteq IL$. Next, we shall prove that $K\cap L$ is PSD in
$L$. Let $K\cap L+X=L, X\leq L$, then $M=L+H=K\cap L+X+H$. Since
$K\cap L$ be PSD in $K$, $K\cap L$ be PSD in $M$ by Lemma 2.2. Thus
there is a projective summand $Y$ of $M$ with $Y\subseteq K\cap L$
such that $M=Y\oplus (X+H)$. Since $L=L\cap M=L\cap (Y\oplus
(X+H))=Y\oplus X+L\cap H$ and $L\cap H$ is PSD in $L$, there is a
projective summand $Y'$ of $L$ with $Y'\subseteq L\cap H$ such that
$L=Y\oplus X\oplus Y'$. Since $L/X\cong Y\oplus Y'$ is projective,
the natural epimorphism $f: K\cap L\rightarrow L/X$ splits, and
hence $Kerf=K\cap X$ is a direct summand of $K\cap L$. Write $K\cap
L=(K\cap X)\oplus Q, Q\leq K\cap L$. So $L=Q\oplus X$, as required.

\end{proof}

\begin{lem} Let $M$ be a $\pi$-projective module. If $N$ and $K$ are
$I$-supplement of each other in $M$, then $N\cap K$ is projective.
If in addition $M$ is projective, then $N$ and $K$ are projective.
\end{lem}
\begin{proof} Let $f: N\oplus K\rightarrow N+K=M$ with $(n,
k)\mapsto n+k$ for $n\in N, k\in K$. Since $M$ is a $\pi$-projective
module, $f$ splits, and so $Kerf=\{(n, -n)| n\in N\cap K\}$ is a
direct summand of $N\oplus K$. Write $N\oplus K=Kerf\oplus U, U\cong
M$. Since $N\cap K$ is PSD in $N$ and $K$, $Kerf$ is PSD in $N\oplus
K$ by Lemma 2.2. Thus there is a projective summand $Y$ of $N\oplus
K$ with $Y\subseteq Kerf$ such that $N\oplus K=Y\oplus U$, so
$Y=Kerf\cong N\cap K$ is projective. If $M$ is projective, $Y\oplus
U$ is projective. So $N$ and $K$ are projective.
\end{proof}

We end this section with the following lemma.

\begin{lem} Let $M=A+B$. If $M/A$ has a projective $I$-cover, then
$B$ contains an $I$-supplement of $A$.
\end{lem}
\begin{proof} Let $\pi: B\rightarrow M/A$ be the canonical
homomorphism and $f: P\rightarrow M/A$ a projective $I$-cover. Since
$P$ is projective, there is a homomorphism $g: P\rightarrow B$ such
that $\pi g=f$. Thus $M=A+g(P)$ and $A\cap g(P)=g(Kerf)$. Since
$Kerf\subseteq IP$ and $Kerf$ is PSD in $P$, $A\cap g(P)\subseteq
Ig(P)$ and $A\cap g(P)$ is PSD in $g(P)$ by Lemma 2.2. So $g(P)$ is
an $I$-supplement of $A$ contained in $B$.
\end{proof}

\section{\bf Characterizations of $I$-semiregular, $I$-semiperfect and $I$-perfect rings in terms of $I$-supplemented modules}

We shall chracterize $I$-semiregular rings, $I$-semiperfect rings
and $I$-perfect rings by $I$-supplemented modules in this section.
We begin this section with the following.

\begin{thm} Let $R$ be a ring and $I$ an ideal of $R$, $P$ a projective
module. Consider the following conditions:
\begin{enumerate}
\item $P$ is an $I$-supplemented module.
\item $P$ is an $I$-semiperfect module.
\end{enumerate}
Then $(1)\Rightarrow (2)$,  and $(2)\Rightarrow (1)$ if $P$ is PSD
for $I$.
\end{thm}
\begin{proof} $``(1)\Rightarrow (2)"$ Let $P$ be an $I$-supplemented
module and $N\leq P$. Then there exists $X\leq P$ such that $P=N+X$,
$N\cap X\subseteq IX$ and $N\cap X$ is PSD in $X$. Let $\pi:
P\rightarrow P/N$ and $\pi\mid_{X}: X\rightarrow P/N$ be the
canonical epimorphisms. Since $P$ is projective, there is a
homomorphism $g: P\rightarrow X$ such that $\pi\mid_{X}g=\pi$. We
have $P=g(P)+N$ and $X=g(P)+N\cap X$. Since $N\cap X$ is PSD in $X$,
there is a projective summand $Y$ of $X$ with $Y\subseteq N\cap X$
such that $X=g(P)\oplus Y$. It is easy to verify that $g(P)\cap
N\subseteq Ig(P)$. Since $g(P)\cap N\subseteq N\cap X$ and $N\cap X$
is PSD in $X$, $g(P)\cap N$ is PSD in $X$ by Lemma 2.2, and so
$g(P)\cap N$ is PSD in $g(P)$ by Lemma 2.2. Thus $g(P)$ is an
$I$-supplement of $N$ in $P$. Since $P$ is an $I$-supplemented
module, $g(P)$ has an $I$-supplement $Q$ in $P$. Thus $g(P)$ is also
an $I$-supplement of $Q$ in $P$ by Lemma 2.9, and so $g(P)$ is
projective by Lemma 2.10. Since $g(P)\cap N\subseteq Ig(P)$ and
$g(P)\cap N$ is PSD in $g(P)$, the canonical epimorphism
$g(P)\rightarrow P/N$ is a projective $I$-cover of $P/N$. So $P$ is
an $I$-semiperfect module by [12, Lemma 2.9].

$``(2)\Rightarrow (1)"$ Let $P$ be an $I$-semiperfect module, then
for every submodule $X$ of $P$, there is a decomposition $P=A\oplus
Y$ such that $A$ is projective, $A\subseteq X$ and $X\cap Y\subseteq
IM$. Thus $P=X+Y$,  $X\cap Y\subseteq IY$. Since $P$ is PSD for $I$,
$X\cap Y$ is PSD in $Y$ by Lemma 2.2, as desired.

\end{proof}

By Theorem 3.1, we know that if a module $M$ is projective and PSD
for $I$, then $M$ is an $I$-supplemented module if and only if $M$
is $I$-lifting if and only if $M$ is an $I$-semiperfect module.

\begin{cor} Let $M$ be a projective
module with $Rad(M)\ll M$ ($\delta(M)\ll_{\delta}M$). Then
 $M$ is a ($\delta$-)supplemented module if and only if
 $M$ is a ($\delta$-)semiperfect module if and only if
 $M$ is a ($\delta$-)lifting module.
\end{cor}

\begin{thm} Let $I$ be an ideal of $R$. Consider the following
conditions:
\begin{enumerate}
\item Every finitely generated $R$-module is $I$-supplemented.
\item Every finitely generated projective $R$-module is
$I$-supplemented.
\item Every finitely generated projective $R$-module is $I$-lifting.
\item $_{R}R$ is an $I$-lifting.

\item $_{R}R$ is an $I$-supplemented.

\item $R$ is $I$-semiperfect.

\end{enumerate}
Then $(1)\Rightarrow (2)\Rightarrow (5)\Rightarrow (6)$ and
$(3)\Rightarrow (4)\Rightarrow (5)$; $(2)\Rightarrow (3)$ and
$(6)\Rightarrow (1)$ if $R$ is a left PSD ring for $I$.

\end{thm}
\begin{proof} $``(1)\Rightarrow (2)\Rightarrow (5)"$ and
$``(3)\Rightarrow (4)\Rightarrow (5)"$ are clear.

$``(5)\Rightarrow (6)"$ By Theorem 3.1.

If $R$ is a left PSD ring for $I$, then $(2)\Rightarrow (3)$ is
obvious by Theorem 3.1 and [12, Corollary 2.4].

$``(6)\Rightarrow (1)"$ Let $M$ be a finitely generated module and
$N\leq M$. Then $M=N+M$ and $M/N$ has a projective $I$-cover by [12,
Theorem 2.13], so $M$ contains an $I$-supplement of $N$ by Lemma
2.11. Hence $M$ is $I$-supplemented.
\end{proof}

Let $I=J(R)$ or $\delta(_{R}R)$ in Theorem 3.3, since $R$ is a left
PSD ring, we have the following.

\begin{cor} ([6, Theorem 4.41]) The following statements are equivalent for a ring $R$.
\begin{enumerate}
\item $R$ is semiperfect.
\item Every finitely generated $R$-module is supplemented.
\item Every finitely generated projective $R$-module is
supplemented.
\item Every finitely generated projective $R$-module is lifting.
\item $_{R}R$ is lifting.

\item $_{R}R$ is supplemented.
\end{enumerate}
\end{cor}

\begin{cor} ([5, Theorem 3.3]) The following statements are equivalent for a ring $R$.
\begin{enumerate}
\item $R$ is $\delta$-semiperfect.
\item Every finitely generated $R$-module is $\delta$-supplemented.
\item Every finitely generated projective $R$-module is
$\delta$-supplemented.
\item Every finitely generated projective $R$-module is $\delta$-lifting.
\item $_{R}R$ is $\delta$-lifting.

\item $_{R}R$ is $\delta$-supplemented.
\end{enumerate}
\end{cor}

Since if $R$ is $Z(_{R}R)$-semiregular, then $Z(_{R}R)=J(R)\subseteq
\delta(_{R}R)$, we have the following result.

\begin{cor} The following statements are equivalent for a ring $R$.
\begin{enumerate}
\item $R$ is $Z(_{R}R)$-semiperfect.
\item Every finitely generated $R$-module is $Z(_{R}R)$-supplemented.
\item Every finitely generated projective $R$-module is
$Z(_{R}R)$-supplemented.
\item Every finitely generated projective $R$-module is $Z(_{R}R)$-lifting.
\item $_{R}R$ is $Z(_{R}R)$-lifting.

\item $_{R}R$ is $Z(_{R}R)$-supplemented.
\end{enumerate}
\end{cor}

Since if $I\leq Soc(_{R}R)$, then $R$ is a left PSD ring for $I$,
and hence we have

\begin{cor} The following statements are equivalent for a ring $R$.
\begin{enumerate}
\item $R$ is $Soc(_{R}R)$-semiperfect.
\item Every finitely generated $R$-module is $Soc(_{R}R)$-supplemented.
\item Every finitely generated projective $R$-module is
$Soc(_{R}R)$-supplemented.
\item Every finitely generated projective $R$-module is $Soc(_{R}R)$-lifting.
\item $_{R}R$ is $Soc(_{R}R)$-lifting.

\item $_{R}R$ is $Soc(_{R}R)$-supplemented.
\end{enumerate}
\end{cor}

\begin{thm} Let $R$ be a left PSD ring and $I$ an ideal of $R$. Then $R$ is an
$I$-semiregular ring if and only if $_{R}R$ is a finitely
$I$-supplemented module if and only if $R_{R}$ is a finitely
$I$-supplemented module.
\end{thm}
\begin{proof} Similar to Theorem 3.3.
\end{proof}

\begin{cor} ([11, Proposition 19.1]) The following statements are equivalent for a ring $R$.
\begin{enumerate}
\item $R$ is  semiregular.
\item $_{R}R$ is a finitely  supplemented module.
\item $R_{R}$ is a finitely  supplemented module.
\end{enumerate}
\end{cor}

\begin{cor} The following statements are equivalent for a ring $R$.
\begin{enumerate}
\item $R$ is $\delta$-semiregular.
\item $_{R}R$ is a finitely $\delta$-supplemented module.
\item $R_{R}$ is a finitely $\delta$-supplemented module.
\end{enumerate}
\end{cor}

\begin{cor} A ring $R$ is $Soc(_{R}R)$-semiregular if and
only if $_{R}R$ is a  finitely $Soc(_{R}R)$-supplemented module if
and only if $R_{R}$ is a finitely $Soc(R_{R})$-supplemented module.
\end{cor}

\begin{cor} A ring $R$ is $Z(_{R}R)$-semiregular if and
only if $_{R}R$ is a  finitely $Z(_{R}R)$-supplemented module if and
only if $R_{R}$ is a finitely $Z(R_{R})$-supplemented module.
\end{cor}

Next we use $I$-supplemented modules to characterize $I$-perfect
rings.

\begin{defn} A ring $R$ is called a strongly left PSD ring for $I$ if any
projective left $R$-module is PSD for $I$.
\end{defn}

\begin{thm} Let $I$ be an ideal of $R$. Consider the following
conditions:
\begin{enumerate}
\item Every $R$-module is $I$-supplemented.
\item Every projective $R$-module is
$I$-supplemented.
\item Every projective $R$-module is
$I$-lifting.
\item Every free $R$-module is
$I$-lifting.

\item Every free $R$-module is $I$-supplemented.

\item $R$ is a left $I$-perfect.

\end{enumerate}
Then $(1)\Rightarrow (2)\Rightarrow (5)\Rightarrow (6)$,
$(3)\Rightarrow (4)\Rightarrow (5)$; $(2)\Rightarrow (3)$ and
$(6)\Rightarrow (1)$ if $R$ is a strongly left PSD ring for $I$.

\end{thm}

\begin{proof} $``(1)\Rightarrow (2)\Rightarrow (5)"$ and
$``(3)\Rightarrow (4)\Rightarrow (5)"$ are clear.

$``(5)\Rightarrow (6)"$ By Theorem 3.1.

When $R$ is a strongly left PSD ring for $I$, $``(2)\Rightarrow
(3)"$ is obvious.

$``(6)\Rightarrow (1)"$ Let $M$ be a module. Then there is a free
module $F$ such that $\eta: F\rightarrow M$ is epic. Since $F$ is
$I$-semiperfect, there is a decomposition $F=F_{1}\oplus F_{2}$ such
that $F_{1}\subseteq Ker\eta$ and $F_{2}\cap Ker\eta\subseteq
IF_{2}$. Since $F$ is PSD for $I$, $F_{2}\cap Ker\eta$ is PSD in
$F$. By Lemma 2.2, $F_{2}\cap Ker\eta$ is PSD in $F_{2}$, so
$\eta|_{F_{2}}: F_{2}\rightarrow M$ is a projective $I$-cover of
$M$. The rest is similar to Theorem 3.3.
\end{proof}

Let $I=J(R)$ or $\delta(_{R}R)$ in Theorem 3.14, since $R$ is a
strongly left PSD ring, we have the following.

\begin{cor} ([6, Theorem 4.41])The following statements are equivalent for a ring $R$.
\begin{enumerate}
\item $R$ is a left perfect.
\item Every $R$-module is supplemented.
\item Every projective $R$-module is
supplemented.
\item Every projective $R$-module is
lifting.
\item Every free $R$-module is
lifting.

\item Every free $R$-module is supplemented.
\end{enumerate}
\end{cor}

\begin{cor} ([5, Theorem 3.4]) The following statements are equivalent for a ring $R$.
\begin{enumerate}
\item $R$ is a left $\delta$-perfect.
\item Every $R$-module is $\delta$-supplemented.
\item Every projective $R$-module is
$\delta$-supplemented.
\item Every projective $R$-module is
$\delta$-lifting.
\item Every free $R$-module is
$\delta$-lifting.

\item Every free $R$-module is $\delta$-supplemented.
\end{enumerate}
\end{cor}

Since if $I\leq Soc(_{R}R)$, then $R$ is a strongly left PSD ring
for $I$, and hence we have

\begin{cor} The following statements are equivalent for a ring $R$.
\begin{enumerate}
\item $R$ is a left $Soc(_{R}R)$-perfect.
\item Every $R$-module is $Soc(_{R}R)$-supplemented.
\item Every projective $R$-module is
$Soc(_{R}R)$-supplemented.
\item Every projective $R$-module is
$Soc(_{R}R)$-lifting.
\item Every free $R$-module is
$Soc(_{R}R)$-lifting.

\item Every free $R$-module is $Soc(_{R}R)$-supplemented.
\end{enumerate}
\end{cor}

Since if $R$ is $Z(_{R}R)$-perfect, then $Z(_{R}R)=J(R)\subseteq
\delta(_{R}R)$, we have the following result.

\begin{cor} The following statements are equivalent for a ring $R$.
\begin{enumerate}
\item $R$ is a left $Z(_{R}R)$-perfect.
\item Every $R$-module is $Z(_{R}R)$-supplemented.
\item Every projective $R$-module is
$Z(_{R}R)$-supplemented.
\item Every projective $R$-module is
$Z(_{R}R)$-lifting.
\item Every free $R$-module is
$Z(_{R}R)$-lifting.

\item Every free $R$-module is $Z(_{R}R)$-supplemented.
\end{enumerate}
\end{cor}


\begin{thebibliography}{99}

\bibitem{A} M. Alkan, W. K. Nicholson, A. C. Ozcan, A generalization
of projective covers, J. Algebra 319 (2008) 4947-4960.
\bibitem{A} M. Alkan,  A. C. Ozcan, Semiregular modules with respect
to a fully invariant submodule, Comm. Algebra 32(11) (2004)
4285-4301.




\bibitem{AF}F. W. Anderson and K. R. Fuller, Rings and Categories
of Modules, Springer-Verlag, New York, 1974.
\bibitem{8} J. Clark, C. Lomp,  N. Vanaja and R. Wisbauer, Lifting
modules, IN:Supplements and Projectivity in Module Theory Sereis:
Frontiers in Mathematics, 2006.


\bibitem{K} M. T. Kosan, $\delta$-lifting and
$\delta$-supplemented modules, Algebra Colloq. 14 (2007) 53-60.

\bibitem{M} S. H. Mohamed and B. J. M\"{u}ller,
Continuous and Discrete Modules, London Math. Soc.; LNS147,
Cambridge Univ. Press: Cambridge, 1990.




\bibitem{N} W. K. Nicholson, Semiregular modules and rings, Can. J. Math.
5 (1976) 1105-1120.

\bibitem{N} W. K. Nicholson, Strong lifting, J. Algebra 285 (2005)
795-818.

\bibitem{N} W. K. Nicholson, M. F. Yousif, Quasi-Frobenius
Rings, Cambridge Tracts in Math., vol.158, Cambridge University
Press, Cambridge, 2003.

\bibitem{A} A. C. Ozcan, M. Alkan, Semiperfect modules with respect
to a preradical, Comm. Algebra 34 (2006) 841-856.


\bibitem{T} A. Tuganbaev, Rings Close to Regular, Kluwer Academic
Publishers, Dordrecht, 2002. ¡¡


\bibitem{W2} Y. D. Wang, Characterizations of $I$-semiregular and
$I$-semiperfect rings, submitted.


\bibitem{W4} Y. D. Wang, Relatively semiregular modules, Algebra
Colloq. 13 (2006) 647-654.


\bibitem{W5} R. Wisbauer, Foundations of Module and Ring
Theory, Gordon and Breach, Philadelphia, 1991.


\bibitem{Z} M. F. Yousif, Y. Q. Zhou, Semiregular,semiperfect and
perfect rings relative to an ideal, Rocky Mountain J. Math. 32
(2002) 1651-1671.


\bibitem{Z} Y. Q. Zhou, Generalizations of perfect, semiperfect,
and semiregular rings, Algebra Colloq. 3 (2000) 305-318.

\end{thebibliography}
\end{document}